\documentclass[onefignum,onetabnum]{siamart171218}
\usepackage{hyperref,booktabs}
\usepackage{amsmath,amssymb,xfrac}
\usepackage[ruled,algo2e]{algorithm2e}

\newcommand\f{\mathsf{F}}
\newcommand\R{\mathbb{R}}
\newcommand\C{\mathbb{C}}
\newcommand\F{\mathbb{F}}
\newcommand\E{\mathbb{E}}
\renewcommand\P{\mathbb{P}}
\newcommand\A{\boldsymbol{A}}
\newcommand\K{\boldsymbol{K}}

\newcommand\M{\boldsymbol{M}}
\renewcommand\S{\boldsymbol{S}}
\newcommand\PP{\boldsymbol{P}}
\newcommand\EE{\boldsymbol{E}}
\newcommand\Q{\boldsymbol{Q}}
\newcommand\U{\boldsymbol{U}}
\newcommand\I{\boldsymbol{I}}

\renewcommand\b{\boldsymbol{b}}
\newcommand\e{\boldsymbol{e}}
\newcommand\x{\boldsymbol{x}}
\newcommand\y{\boldsymbol{y}}
\newcommand\z{\boldsymbol{z}}
\newcommand\q{\boldsymbol{q}}
\newcommand\aalpha{\boldsymbol{\alpha}}
\newcommand\0{\boldsymbol{0}}
\newcommand\cur{\mathsf{cur}}

\newcommand\nnz{\operatorname{\textsc{nnz}}}
\newcommand\tr{\operatorname{tr}}

\newcommand\Blk{\ensuremath{\boldsymbol{D}}}
\newcommand\RCM{\ensuremath{\boldsymbol{R}}}
\newcommand\Stab{\operatorname{Stab}}

\ifpdf
\hypersetup{
  pdftitle={A Randomized Algorithm for Preconditioner Selection},
  pdfauthor={D. Conner, G. Weiqing}
}
\fi

\title{A Randomized Algorithm for\\ Preconditioner Selection
}
\author{Conner DiPaolo\thanks{Author was with the Department of Mathematics, Harvey Mudd College during the progress of this research. He is now with Yelp Inc., San Francisco, California (\url{conner@yelp.com}).} \and Weiqing Gu\thanks{Department of Mathematics, Harvey Mudd College, Claremont, California (\url{gu@hmc.edu}).}}

\headers{A Randomized Algorithm for Preconditioner Selection}{Conner DiPaolo and Weiqing Gu}

\begin{document}

\maketitle

% REQUIRED
\begin{abstract}
    The task of choosing a preconditioner $\M$ to use when solving a linear system $\A\x=\b$ with iterative methods is difficult. For instance, even if one has access to a collection $\M_1,\M_2,\ldots,\M_n$ of candidate preconditioners, it is currently unclear how to practically choose the $\M_i$ which minimizes the number of iterations of an iterative algorithm to achieve a suitable approximation to $\x$. This paper makes progress on this sub-problem by showing that the preconditioner stability $\|\I-\M^{-1}\A\|_\mathsf{F}$, known to forecast preconditioner quality, can be computed in the time it takes to run a constant number of iterations of conjugate gradients through use of sketching methods. This is in spite of folklore which suggests the quantity is impractical to compute, and a proof we give that ensures the quantity could not possibly be approximated in a useful amount of time by a \emph{deterministic} algorithm. Using our estimator, we provide a method which can provably select the minimal stability preconditioner among $n$ candidates using floating point operations commensurate with running on the order of $n\log n$ steps of the conjugate gradients algorithm. Our method can also advise the practitioner to use no preconditioner at all if none of the candidates appears useful. The algorithm is extremely easy to implement and trivially parallelizable. In one of our experiments, we use our preconditioner selection algorithm to create to the best of our knowledge the first preconditioned method for kernel regression reported to never use more iterations than the non-preconditioned analog in standard tests.
\end{abstract}

% REQUIRED

% REQUIRED

\section{Introduction}

A preconditioner $\M$ is helpful for solving a linear system $\A\x=\b$ via iterative methods if it reduces the number of iterations enough to offset the cost of constructing the preconditioner plus the additional cost of taking the iterations (typically an extra computation of the form $\M^{-1}\z$ per iteration.) Ignoring the latter component of this trade-off, this corresponds to having a small condition number $\kappa(\M^{-1}\A)$ for the conjugate gradient method \cite{trefethen1997numerical,deift2019conjugate}; a precise objective is unclear in the indefinite or unsymmetric case with general Krylov methods besides the general desire that the spectrum of $\M^{-1}\A$ be clustered \cite{benzi1999orderings}. For instance, even if we have a small number of candidate preconditioners $\M_1,\M_2,\ldots,\M_n$ for our problem ready to use and assume that they add the same amount of time to compute each iteration, it is unclear how one would go about estimating which preconditioner would reduce the iteration count the most without actually solving a system with each preconditioner or doing a comparable amount of work. This preconditioner selection task is the focus of the present work.

\subsection{Contributions}
\label{sub:contrib}

The core contribution of this work is the realization that randomized sketching methods make completely practical the computation of a helpful forecaster of preconditioner quality previously thought to be infeasible to compute.

In addition to this primary method for computing preconditioner stability, we provide a number of other results:

\begin{itemize}
    \item We prove that no practical deterministic algorithm, in a meaningful sense, could possibly be used to estimate the preconditioner stability $\|\I - \M^{-1}\A\|_\mathsf{F}$.
    \item We confirm the conjecture of \cite{avron2011randomized} regarding the true asymptotic sample complexity of the Gaussian trace estimator using a substantially more direct proof than the general result provided in \cite{wimmer2014optimal}, at the same time confirming the tightness of our stability estimation algorithm convergence bound.
    \item We provide a randomized algorithm which can provably select a preconditioner of approximately minimal stability among $n$ candidate preconditioners using computational resources equivalent to computing on the order of $n\log n$ steps of the conjugate gradients algorithm.
    \item We give a theoretical speedup to the initial preconditioner selection method which largely decouples the runtime dependence between the number of preconditioners $n$ and the desired accuracy in the situation that the input preconditioner stabilities satisfy an anti-concentration assumption.
    \item Using our initial preconditioner selection algorithm, we create the first (to the best of our knowledge) method for preconditioning in kernel regression which is reported to never give a worse number of iterations than using no preconditioner in standard tests.
\end{itemize}

It is important to point out that while our experiments consider positive definite systems and preconditioners, our algorithms and runtime bounds hold as-is for arbitrary matrices $\A$ and preconditioners $\M$.

\subsection{Prior Art}
\label{sub:priorwork}

Current approaches to forecasting preconditioner quality are not robust in the sense that they can fail to select a highly performant preconditioner in favor of a poorly performing one in many realistic scenarios. This lack of robustness, even in the case of positive definite $\A$, results in part from the fact that the research community does not have robust methods for estimating condition numbers of large sparse matrices, which makes proxies for preconditioner quality necessary. For instance, Avron et al. \cite{avron2019spectral} recently produced a condition number estimator in this setting which appears to perform admirably in many situations but does not always converge and at this point does not have rigorous theoretical backing.

The simplest forecasting criterion is that a preconditioner $\M$ ought to be an `accurate' approximation, in the sense that the Frobenius norm distance $\|\M - \A\|_\f$ is small. This preconditioner accuracy criterion $\|\M-\A\|_\f$ can be efficiently computed so long as one has access to the preconditioning matrix $\M$ in memory as well as the standard matrix-vector product access to $\M^{-1}$. Moreover, for symmetric $M$-matrices this accuracy criterion is a useful proxy for the number of conjugate gradient iterations necessary to solve the preconditioned system in $\A$. This point was theoretically confirmed by Axelsson and Eijkhout \cite{axelsson1990vectorizable}, who showed that the condition number $\kappa(\M^{-1}\A)$ can be bounded in terms of $\|\M^{-1}\|_\f\|\M-\A\|_\f$. The accuracy criterion for forecasting preconditioner quality was heavily tested on an empirical level in \cite{duff1989effect}.

Even in this setting, though, there exist accurate real-world preconditioners that give a poor iteration count because $\|\M^{-1}\|_\f$ is very large \cite{benzi1999orderings}. Unfortunately, to detect this issue one presumably needs to compute the entire matrix $\M^{-1}$ in memory solely from matrix-vector product access to $\M^{-1}$. The natural method for this, computing each of the $d$ columns $\M^{-1}\e_i$ from products with the standard basis vectors $\e_i$, takes the same leading-order floating point cost as solving the system $\A\x=\b$ via conjugate-gradients, at least when matrix-vector multiply access to $\M^{-1}$ is comparable to the matrix-vector multiply cost of multiplying by $\A$. Perhaps as a result of this intuition, computing $\|\M^{-1}\|_\f$, or rather its close cousin the $\ell^\infty$-norm $\|\M^{-1}\|_\infty$, was deemed impractical \cite{benzi1999orderings,benzi2002preconditioning}. In practice, Chow and Saad \cite{chow1997experimental} suggested using the lower bound $\|\M^{-1}\e\|_\infty$ for $\|\M^{-1}\|_\infty$ where $\e$ is the vector of all ones. This has proved helpful for detecting preconditioner instability in some situations, and especially indefinite systems \cite{benzi1999orderings}.

The quantity known as `preconditioner stability,' $\|\I - \M^{-1}\A\|_\f$, is empirically the most reliable indicator of preconditioner performance from this class of forecasts. This is especially true among many non-symmetric problems or problems which are far from diagonally dominant \cite{benzi1999orderings}. Unfortunately, prior work has suggested that computing preconditioner stability is `impractical' \cite{benzi2002preconditioning} for effectively the same reason as why $\|\M^{-1}\|_\infty$ was considered impractical to compute.

In spite of this, Tyrtyshnikov's influential paper \cite{tyrtyshnikov1992optimal} gave a fast algorithm to solve the harder problem considered in this paper of determining the minimal-stability preconditioner over the special class of \emph{circulant} matrices. This algorithm relied heavily on properties of these matrices.

Our results in this paper will suggest that even for generic $\A$ and $\M$, one can get an arbitrarily accurate estimate of most of these quantities previously thought to be unusable for preconditioner selection. That said, it is crucial to note that the prior impracticality assessment is completely valid in a certain regard: no deterministic algorithms, as we also show, could possibly even approximate the preconditioner stability, for example. Thus an algorithm must fail at some of these tasks with some nonzero probability.

\subsection{Overview}
\label{sub:overview}

Section \ref{sec:alg} motivates the need for a randomized algorithm for stability estimation with theory, responds with a sketching-based solution, and uses it to create and analyze a method for preconditioner selection. This theory is empirically confirmed in Section \ref{sec:experiments} where we apply the primary preconditioner selection algorithm from Section \ref{sec:alg} to solving generic real-world systems (Section \ref{sec:experiments_sparse}) and creating more robust preconditioned solvers for kernel regression (Section \ref{sec:experiments_kernel}.) After this we conclude in Section \ref{sec:conclusions} and discuss open problems that could provide an avenue for future work.

\subsubsection{Notation}
\label{sub:notation}

Boldface letters like $\x,\y,\b$ denote vectors while their upper-case analogues such as $\A,\PP,\Q$ denote matrices. Such matrices and vectors can be either random or deterministic. The underlying scalar field $\F$ is either the real numbers $\R$ or the complex numbers $\C$ unless otherwise specified. The adjoint of a matrix is denoted by $\A^*$. The inner product of two vectors $\x$ and $\y$ is denoted as $\x^*\y$. The expectation of a random variable will be denoted with $\E$, and the probability of an event $E$ is written $\P(E)$. The norms $\|\A\|_\f$ and $\|\x\|_2$ represent the Frobenius and Euclidean norms, respectively. The condition number $\kappa(\A) = \sigma_{\max}(\A) / \sigma_{\min}(\A)$ is the $\ell^2$ condition number. The notation $a_n \sim b_n$ represents $\lim_{n\to\infty} \tfrac{a_n}{b_n} = 1$. $\mathcal{N}(\boldsymbol{\mu},\boldsymbol{\Sigma}) = \mathcal{N}_\R(\boldsymbol{\mu},\boldsymbol{\Sigma})$ and $\mathcal{N}_\C(\boldsymbol{\mu},\boldsymbol{\Sigma})$ are normal and circularly-symmetric complex normal random variables, respectively.

\section{Algorithms}
\label{sec:alg}

This section contains all of our core theoretical results and algorithms. In Section \ref{sub:randomization_necessary} we show that the only algorithms which can possibly estimate preconditioner stability must be randomized. The follow-up question of whether randomization can indeed work is answered in the affirmative in Section \ref{sub:computing_stability}, where we show that a slight adaptation of a well-known sketching-based algorithm for computing Schatten norms perfectly fits our realistic access model to our preconditioner $\M$ and matrix $\A$. Once we have a good estimator of preconditioner stability, a natural method for selecting the candidate preconditioner with minimal stability criterion presents itself in Section \ref{sub:thealgorithm}. It turns our that our algorithm can be trivially parallelized, and a testament to this fact is given in Section \ref{sub:parallel}. In Section \ref{sub:upper_bound} we take advantage of highly informative results from the literature on trace estimation to provide useful approximation guarantees and runtime bounds for the previously presented algorithms. Section \ref{sec:lower_bound} returns the focus to Theorem \ref{thm:stab} of Section \ref{sub:upper_bound}, showing that the leading constant is tight and providing a more concise resolution of a conjecture from \cite{avron2011randomized} than the more generalized result from \cite{wimmer2014optimal}. In that section, we also comment that no randomized algorithm for estimating preconditioner stability which has access to matrix-vector products of the form $(\I-\M^{-1}\A)\z$ could possibly do better asymptotically than Algorithm \ref{alg:stab} by relying on this result \cite[Thm. 3]{wimmer2014optimal} from the trace estimation literature. Finally, this conversation is wrapped up in Section \ref{sub:winner}, where our bounds from Section \ref{sub:upper_bound} are utilized to provide a theoretical improvement to Algorithm \ref{alg:pick}.

\subsection{Randomization is Necessary to Compute Preconditioner Stability}
\label{sub:randomization_necessary}

This paper provides a simple randomized algorithm which can accurately estimate the preconditioner stability $\|\I-\M^{-1}\A\|_\f$ in time faster than running a constant number of iterations of preconditioned conjugate gradients with the matrix $\A$ and preconditioner $\M$. Through incorporating randomness, however, we must accept that the algorithm fails with some probability. This failure probability can be made arbitrarily small, but it would still be advantageous (for example, in mission-critical applications) to provide a deterministic algorithm for the same task with a comparable approximation guarantee. The purpose of this section is to crush that latter hope, and the following theorem does just that.

\begin{theorem}\label{thm:impossible}
    Fix some $0 \leq \epsilon < 1$. Suppose we have a deterministic algorithm which takes as input an arbitrary positive semi-definite matrix $\A\in\F^{d\times d}$ and positive definite matrix $\M\in\F^{d\times d}$, and returns an estimate $\textsc{Alg}(\A,\M)$ satisfying $$(1-\epsilon)\|\I - \M^{-1}\A\|_\f \leq \textsc{Alg}(\A,\M) \leq (1+\epsilon)\|\I - \M^{-1}\A\|_\f$$ after sequentially querying and observing matrix vector multiplies of the form\\ $(\M^{-1}\q_i,\A\q_i)$ for $i=1,2,\ldots,k$ where $k$ is a constant depending only on $d$ and $\epsilon$. Then $k\geq d$.
\end{theorem}
\begin{proof}
    Fix $\M=\I$ for the remainder of the proof. Suppose to the contrary that $k = d-1$ suffices to compute $\textsc{Alg}(\A,\M)$. Let $\q_1,\q_2,\ldots,\q_{d-1}$ be the query vectors used by the algorithm in the case that $(\I - \M^{-1}\A)\q_i$ always returns $\0$. Write $\PP$ for the orthogonal projection onto $\operatorname{span}\{\q_1,\q_2,\ldots,\q_{d-1}\}$. Both of the positive semi-definite matrices $\A=\I$ and $\A=\PP$ will return $(\I - \M^{-1}\I)\q_i=(\I-\M^{-1}\PP)\q_i=\0$ uniformly over $i=1,2,\ldots,d-1$, and thus since the algorithm is deterministic the estimated stabilities $\textsc{Alg}(\I,\M)=\textsc{Alg}(\PP,\M)$ are equal. But $\PP \neq \I$ since $\PP$ was an orthogonal projection onto a subspace of dimension strictly less than $d$, and hence
    \[
        0 < (1-\epsilon)\|\I-\PP\|_\f \leq \textsc{Alg}(\PP,\M) = \textsc{Alg}(\I,\M) \leq (1+\epsilon)\|\I-\I\|_\f = 0
    \]
    by our approximation guarantee. This contradiction ensures that we must take $k \geq d$. Since $\M=\I$, matrix-vector multiply access to $\I - \M^{-1}\A$ is equivalent via a bijection to matrix-vector multiply access to $\M^{-1}$ and $\A$, hence our strengthened statement of the result.
\end{proof}

Of course, using $k=d$ queries suffices to achieve no error at all, and so the above lower bound is tight:
\begin{equation}\label{eq:deterministic}
    \|\I - \M^{-1}\A\|_\f = \biggl(\sum_{i=1}^d \|(\I-\M^{-1}\A)\e_i\|_2^2\biggr)^{1/2}
\end{equation}
where $\e_1,\e_2,\ldots,\e_d$ is any orthonormal basis for $\F^d$. Also, note that the condition that $\A$ and $\M$ be positive semi-definite gives a stronger result than if they were allowed to be arbitrary matrices.

In order to put Theorem \ref{thm:impossible} into better context, recall that the dominant cost of an iteration of preconditioned conjugate gradients \cite[Alg. 11.5.1]{golub2012matrix} is (a) computing $\A\y$ for a vector $\y$, and (b) computing $\M^{-1}\z$ for a vector $\z$. Thus Theorem \ref{thm:impossible} says roughly that in the time it takes to even approximate $\|\I-\M^{-1}\A\|_\f$ deterministically, one can solve a system $\A\x=\b$ exactly (in exact arithmetic) by running the conjugate gradients algorithm for $d$ iterations. Since our whole goal of computing $\|\I-\M^{-1}\A\|_\f$ is to forecast how well $\M$ would do as a preconditioner for solving the system $\A\x=\b$, this means that any deterministic algorithm for this task is, in general, impractical for this task.

\subsection{Computing Preconditioner Stability via Randomization}
\label{sub:computing_stability}

\begin{algorithm2e}[t]
    \KwData{A matrix $\A\in\F^{d\times d}$, preconditioner $\M\in\F^{d\times d}$, and an accuracy parameter $k\in\{1,2,\ldots\}$.}
    \KwResult{A estimate of the preconditioner stability $\|\I - \M^{-1}\A\|_\f$.}

    Form a matrix $\Q = [\q_1,\ldots,\q_k]$ with independent columns $\q_i \sim \mathcal{N}_\F(\0,\tfrac{1}{k}\I_d)$.

    Construct the sketch $\S = (\I - \M^{-1}\A)\Q$ via its columns $\q_i - \M^{-1}(\A \q_i)$.

    Return $\|\S\|_\f$.

    \caption{$\Stab(\A,\M,k)$: Estimates the stability of the preconditioner $\M$ in $\sim 3dk + kT_{m} + 2k\nnz(\A)$ floating point operations when $\F=\R$ or $\sim 6dk + kT_{m} + 8k\nnz(\A)$ flops when $\F=\C$, where $T_{m}$ is the number of flops needed to compute $\M^{-1}\b$ for an arbitrary $\b\in\F^d$.}\label{alg:stab}
\end{algorithm2e}

Now we will demonstrate that, unlike the deterministic case, randomization makes it entirely practical to compute preconditioner stability. To see why this is intuitive, let $\q$ be a standard (real or complex) Gaussian vector. Then
\begin{align}
    \|\I - \M^{-1}\A\|_\f^2 &= \tr\bigl((\I - \M^{-1}\A)^*(\I-\M^{-1}\A)\bigr)\\
    &= \tr\bigl((\I - \M^{-1}\A)^*(\I-\M^{-1}\A)\E\q\q^*\bigr)\\
    &= \E\tr\bigl(\q^*(\I - \M^{-1}\A)^*(\I-\M^{-1}\A)\q\bigr)\\
    &= \E \|(\I - \M^{-1}\A)\q\|_2^2
\end{align}
by the linearity of expectation, the cyclic property of the trace, and the fact that $\E\q\q^* = \I$. Thus, if $\q_i$ are independent standard normal vectors for all $i=1,2,\ldots,k$, the Monte-Carlo squared stability estimate
\begin{equation}
    S^2 = \frac{1}{k}\sum_{i=1}^k \|(\I - \M^{-1}\A)\q_i\|_2^2 \to \|\I - \M^{-1}\A\|_\f^2
\end{equation}
almost surely as $k\to\infty$ by the strong law of large numbers. We can rewrite the above estimator as
\(
    S = \|(\I - \M^{-1}\A)\Q\|_\f
\)
where $\Q$ is a matrix with independent and identically distributed elements $\Q_{ii}\sim \mathcal{N}(0,\tfrac{1}{k})$. This stability estimation algorithm for $S = \sqrt{S^2} \approx \|\I - \M^{-1}\A\|_\f$ is given as Algorithm \ref{alg:stab}. In Algorithm \ref{alg:stab}, we adjust whether $\q$ is real or complex depending on  $\F$ for analysis reasons that will be apparent later in Section \ref{sub:upper_bound}.

It is important to note that the mathematical foundations of the above algorithm are not novel. It is equivalent in exact arithmetic to applying the trace estimators in \cite{roosta2015improved} to the matrix $(\I-\M^{-1}\A)^*(\I-\M^{-1}\A)$ and then taking the square root. It is also a simplification of the Schatten-$2$ norm estimator in \cite[Thm. 69]{woodruff2014sketching} (relayed from \cite{li2014sketching}) applied to $\I - \M^{-1}\A$. The reason we include Algorithm \ref{alg:stab} is not because of its mathematical novelty but because of its observational novelty: sketching algorithms using the matrix-vector multiply access model are a perfect fit for interrogating the matrices $\M^{-1}$ and $\A$ in the context of iterative methods for solving linear systems, since this kind of access to $\M^{-1}$ and $\A$ are precisely what make that kind of access practical.

Of course, the presentation thus far does not help us choose how large the accuracy parameter $k$ should be. To that end, Theorem \ref{thm:stab} in Section \ref{sub:upper_bound} presents a tight upper bound on the necessary $k$ to achieve a given error bound with high probability, although in practice just setting $k=10$ appears sufficient in many situations of interest -- see Section \ref{sec:experiments}.

\subsection{Randomized Algorithm for Selecting the `Best' Preconditioner}
\label{sub:thealgorithm}

Once we have a practical way to compute preconditioner stability, a trivial algorithm for picking the preconditioner among $n$ candidates $\M_1,\M_2,\ldots,\M_n$ becomes natural. Namely, we can compute estimates $S_i \approx \|\I-\M_i^{-1}\A\|_\f$ for $i=1,2,\ldots,n$ and then just return the preconditioner $\M_i$ for which $S_i$ is minimized. This is presented as Algorithm \ref{alg:pick}. As we mentioned in the previous section, theoretical advice on how to pick $k$ will be given in Section \ref{sub:upper_bound}. An improvement to this algorithm in the case there is a clear winner, relying on those analytical bounds, is included in Section \ref{sub:winner}.

\begin{algorithm2e}[t]
    \KwData{A matrix $\A\in\F^{d\times d}$, $n$ candidate preconditioners $\M_1$, $\M_2$,$\ldots$,$\M_n\in\F^{d\times d}$, and an accuracy parameter $k\in\{1,2,\ldots\}$.}
    \KwResult{A preconditioner $\M_i$ which approximately minimizes the stability criterion $\|\I - \M_j^{-1}\A\|_\f$ over $j=1,\ldots,n$.}

    Compute a stability estimate $S_j = \Stab(\A,\M_j,k)$ for each $j=1,2,\ldots,n$.

    Return an arbitrary $\M_i$ with $\displaystyle S_i = \min_{1\leq j \leq n} S_j$.

    \caption{Returns an approximately optimal preconditioner in the Tyrtyshkinov sense among $n$ candidates with strictly fewer floating point operations than running $k$ iterations of preconditioned conjugate gradients \cite[Alg. 11.5.1]{golub2012matrix} with each of the $n$ preconditioners.}\label{alg:pick}
\end{algorithm2e}

We note that the sketching matrix $\Q$ can be re-used when computing the stability estimates $S_j$ in Algorithm \ref{alg:stab}. This is done in all our computational experiments in Section \ref{sec:experiments}, and reduces the number of normal variates one needs to simulate from $ndk$ to $dk$. Reuse does not affect our theoretical upper bound presented as Theorem \ref{thm:pick}.

\subsubsection{Parallelization}
\label{sub:parallel}

A convenient aspect of sketching based algorithms like Algorithm \ref{alg:pick} is that they can be parallelized extremely easily. For instance, suppose we are trying to pick the minimal stability preconditioner among $n$ candidates $\M_1,\M_2,\ldots,\M_n$, and have $n$ processors $P_j$, $j=1,2,\ldots,n$, which have access to $\A$ and $\M_j$, respectively. Then we can compute each stability estimate $S_j$ in parallel at processor $P_j$. Ignoring communication costs (which are a genuine concern in practice,) this would bring the runtime of the algorithm down to computing $k$ steps of preconditioned conjugate gradients with $\A$ and the most computation-intensive (in terms of matrix-vector multiply access to $\M^{-1}$) preconditioner $\M_j$.

Taken to the extreme, one could similarly parallelize Algorithm \ref{alg:pick} over $nk$ processors $P_{ij}$, $i=1,2,\ldots,n$ and $j=1,2,\ldots,k$, assuming each $P_{ij}$ had access to $\A$ and the candidate preconditioner $\M_j$. Each processor $P_{ij}$ would need to compute and return $s_{ij}=\|(\I - \M_j^{-1}\A)\q_i\|_2^2$ where $\q_i$ is an independently sampled standard normal vector. Then in parallel for all $i=1,2,\ldots,n$ processor $P_{i1}$ could compute $S_i^2 = \tfrac{1}{k}\sum_{j=1}^k s_{ij}$, at which point we could use processor $P_{11}$ to compute $i$ such that $S_i^2$ is minimal and return the corresponding $\M_i$. Ignoring communication costs again, this algorithm take fewer floating point operations than running \emph{one} iteration of preconditioned conjugate gradients with $\A$ and the most computation-intensive preconditioner $\M_j$, plus $k$ flops that were used to turn the $s_{ij}$ into the estimate $S_i$.

\subsection{Approximation Guarantees and Runtime Bounds}
\label{sub:upper_bound}

This section details runtime bounds and approximation guarantees for Algorithms \ref{alg:stab} and \ref{alg:pick}, as well as using those bounds to derive a theoretical improvement to Algorithm \ref{alg:pick}.

To start, we will give the following theorem, which says that to estimate preconditioner stability up to a $1\pm\epsilon$ multiplicative factor with failure probability at most $\delta$, one may take $k = \mathcal{O}(\tfrac{1}{\epsilon^2}\log\tfrac{1}{\delta})$ in Algorithm \ref{alg:stab}. This is (in contrast to the deterministic case) completely independent of the underlying dimension. Theorem \ref{thm:stab} result sharpens an analogous bound for trace estimators given in \cite[Thm. 3]{roosta2015improved}, largely following the proof structure of Theorem 1 from that work. We bring the leading constant from a $8$ to $4$ as $\epsilon\to 0$ and allow for complex matrices. This Theorem is included as a result in its own right primarily since we can even show that the leading constant is tight; see Section \ref{sec:lower_bound}.

\begin{theorem}\label{thm:stab}
    Let $\M$ and $\A$ be arbitrary matrices in $\F^{d\times d}$ where $\M$ is invertible. If $\epsilon$ and $\delta$ are positive and less than one, taking $k \geq \tfrac{12}{\epsilon^2(3-2\epsilon)}\log\tfrac{2}{\delta}$ ensures that in exact arithmetic the estimate $\Stab(\A,\M,k)\in\mathbb{R}$ of Algorithm \ref{alg:stab} satisfies
    \[
        \sqrt{1-\epsilon}\,\|\I-\M^{-1}\A\|_\f \leq \Stab(\A,\M,k) \leq \sqrt{1+\epsilon}\,\|\I-\M^{-1}\A\|_\f.
    \]
    with probability at least $1-\delta$. In particular, if $\epsilon \leq 1/2$, then the simpler condition $k \geq \tfrac{6}{\epsilon^2}\log\tfrac{1}{\delta}$ ensures this same approximation guarantee.
\end{theorem}
\begin{proof}
    Unitarily diagonalize $(\I - \M^{-1}\A)^*(\I - \M^{-1}\A)$ as $\U^* \Lambda \U$ for some unitary $\U$ and non-negative diagonal matrix $\Lambda=\operatorname{diag}(\lambda_1,\ldots,\lambda_d)$. By the rotation invariance of the Gaussian and the fact that $\U$ is orthogonal when $\F=\R$, $\|(\I - \M^{-1}\A)\q_i\|_2$ is equal in distribution to $\|\Lambda^{1/2}\q_i\|_2$ in both the real and complex cases. This implies $\Stab(\A,\M,k)$ is equal in distribution to $\tfrac{1}{k}\sum_{i=1}^k \q_i^*\Lambda\q_i = \tfrac{1}{k}\sum_{j=1}^d \lambda_j \sum_{i=1}^k|\q_{ij}|^2$ and we can focus on the latter, simpler quantity. By Markov's inequality, for any  $t > 0$,
    \begin{align}
        \P\bigl(\Stab(\A,\M,k)^2& > (1+\epsilon)\|\I - \M^{-1}\A\|_\f^2\bigr)\label{eq1}\\ &= \P\biggl(\sum_{j=1}^d \lambda_j \sum_{i=1}^k|\q_{ij}|^2 \geq k(1+\epsilon)\sum_{j=1}^d \lambda_j\biggr)\label{eq2}\\
        &\leq e^{-(1+\epsilon)kt} \E \exp\biggl(\sum_{j=1}^d \frac{\lambda_j}{\sum_{j=1}^d \lambda_j}\sum_{i=1}^k t |\q_{ij}|^2\biggr).\label{eq:bound}
    \end{align}
    Jensen's inequality reduces
    \begin{equation}
        \E \exp\biggl(\sum_{j=1}^d \frac{\lambda_j}{\sum_{j=1}^d \lambda_j}\sum_{i=1}^k t |\q_{ij}|^2\biggr) \leq \prod_{i=1}^k \E e^{t |\q_{i1}|^2} \leq (1-2t)^{-\tfrac{k}{2}}
    \end{equation}
    so long as $t < 1/2$. Note that the last inequality covers both real and complex $\F$, and results from the $\F=\C$ moment generating function $\E e^{t |\q_{i1}|^2} = (1-t)^{-1} \leq (1-2t)^{-1/2}$ being bounded above uniformly by the real case for this range of $t$. Taking $t = \tfrac{1}{2}\tfrac{\epsilon}{1+\epsilon} \in (0,1/2)$ in Equation \ref{eq:bound} gives
    \begin{equation}
        \P\bigl(\Stab(\A,\M,k)^2 > (1+\epsilon)\|\I - \M^{-1}\A\|_\f^2\bigr) \leq \bigl((1+\epsilon)e^{-\epsilon}\bigr)^{m/2} < e^{-\tfrac{m}{2}\bigl(\tfrac{\epsilon^2}{2} - \tfrac{\epsilon^3}{3}\bigr)}
    \end{equation}
    via the scalar inequality $(1+\epsilon)e^{-\epsilon} < e^{\tfrac{\epsilon^3}{3} - \tfrac{\epsilon^2}{2}}$ for all $\epsilon > 0$. The same argument for the lower tail, instead taking $t = \tfrac{1}{2}\tfrac{\epsilon}{1-\epsilon}$, gives
    \begin{equation}
        \P\bigl(\Stab(\A,\M,k)^2 < (1-\epsilon)\|\I - \M^{-1}\A\|_\f^2\bigr) < e^{-\tfrac{m}{2}\bigl(\tfrac{\epsilon^2}{2} - \tfrac{\epsilon^3}{3}\bigr)}
    \end{equation}
    as well. A union bound provides the desired result.
\end{proof}

Using Theorem \ref{thm:stab} we are able to prove an approximation guarantee for Algorithm \ref{alg:pick} via a union bound. In particular, to achieve an $\epsilon$-multiplicative approximation to the best of $n$ candidate preconditioners with probability at least $1-\delta$ we can take $k = \mathcal{O}(\tfrac{1}{\epsilon^2}\log\tfrac{n}{\delta})$, again independent of the underlying dimension. This dependence on $n$ is quite weak, especially since in realistic applications we would only expect to have at most, say, fifty candidate preconditioners.

\begin{theorem}\label{thm:pick}
    Let $\A\in\F^{d\times d}$ be an arbitrary matrix, and $\M_1,\M_2,\ldots,\M_n\in\F^{d\times d}$ be invertible candidate preconditioners for $\A$. If $\epsilon$ and $\delta$ are positive and less than one, taking $k\geq \tfrac{12}{\epsilon^2(3-2\epsilon)}\log\tfrac{2n}{\delta}$ ensures that the preconditioner $\M_i$ returned by Algorithm \ref{alg:pick} satisfies
    \[
        \|\I - \M_i^{-1}\A\|_\f \leq \sqrt{\frac{1+\epsilon}{1-\epsilon}}\,\min_{1\leq j\leq n} \|\I - \M_j^{-1}\A\|_\f
    \]
    with probability at least $1-\delta$. In particular, if $\epsilon < 1/2$ the simpler condition $k \geq \tfrac{11}{\epsilon^2}\log\tfrac{2n}{\delta}$ ensures
    \[
        \|\I - \M_i^{-1}\A\|_\f \leq (1+\epsilon)\,\min_{1\leq j\leq n} \|\I - \M_j^{-1}\A\|_\f
    \]
    with probability at least $1-\delta$.
\end{theorem}
\begin{proof}
    Start by fixing any $j\in\{1,2,\ldots,n\}$. If we take $k\geq \tfrac{12}{\epsilon^2(3-2\epsilon)}\log\tfrac{2n}{\delta}$, Theorem \ref{thm:stab} ensures that
    \begin{equation}\label{eq:approx}
        \sqrt{1-\epsilon}\,\|\I-\M_j^{-1}\A\|_\f \leq \Stab(\A,\M_j,k) \leq \sqrt{1+\epsilon}\,\|\I-\M_j^{-1}\A\|_\f,
    \end{equation}
    except with probability at most $\tfrac{\delta}{n}$. In particular, if we unfix $j$ the probability at least one of the $\Stab(\A,\M_j,k)$ does not satisfy Equation \ref{eq:approx} is at most
    \(
        \sum_{j=1}^n \tfrac{\delta}{n} = \delta
    \)
    by a union bound. (Note that we did not need independence of the estimates $\Stab(\A,\M_j,k)$ here; this is why reusing the sketching matrix $\Q$ is valid.) Thus with probability at least $1-\delta$ all estimates $\Stab(\A,\M_j,k)$ satisfy Equation \ref{eq:approx} simultaneously.
    
    Write $\M_i$ for the candidate preconditioner returned by Algorithm \ref{alg:pick}, and write $\M_\star$ for a candidate preconditioner which satisfies
    \begin{equation}\label{eq:opt}
        \|\I - \M_\star^{-1}\A\|_\f = \min_{1\leq j \leq n}\|\I - \M_j^{-1}\A\|_\f.
    \end{equation}
    Then since the \emph{estimate} of the stability of $\M_i$ was at most that of $\M_\star$ by minimality, the simultaneous bounds of Equation \ref{eq:approx} give
    \[
        \sqrt{1-\epsilon}\|\I - \M_i^{-1}\A\|_\f \leq \Stab(\A,\M_i,k) \leq \Stab(\A,\M_\star,k) \leq \sqrt{1+\epsilon}\|\I - \M_\star^{-1}\A\|_\f
    \]
    except with probability at most $\delta$. Rearranging the inequality gives the desired result after substituting Equation \ref{eq:opt}.

    The final simplified bound results from the scalar inequality $\sqrt{\tfrac{1+\epsilon}{1-\epsilon}} \leq 1+\tfrac{4}{3}\epsilon$ when $0 \leq \epsilon < 2/5$ and simple algebraic manipulation.
\end{proof}

\subsubsection{The Constant in Theorem \ref{thm:stab} is Tight}
\label{sec:lower_bound}

Most of the theory presented in this paper relies on Theorem \ref{thm:stab} to create more sophisticated bounds. Since Algorithm \ref{alg:stab} is at its core a repurposing of a trace estimator using only matrix vector products, the work \cite{wimmer2014optimal} applies and ensures that no randomized, adaptive algorithm for estimating the stability $\|\I - \M^{-1}\A\|_\f^2 = (\I - \M^{-1}\A)^*(\I - \M^{-1}\A)$ could possibly use asymptotically fewer matrix-vector multiplies so long as the algorithm only has access to $\|(\I - \M^{-1}\A)\q\|_2$ for query vectors $\q$. In this sense, Algorithm \ref{alg:stab} is optimal.

The theoretically-inclined practitioner, however, also cares about knowing the optimality of our \emph{analysis} in Theorem \ref{thm:stab}. The following Theorem says that our analysis in Theorem \ref{thm:stab} is asymptotically tight even up to the leading effective constant $12 / (3-2\epsilon)$ which tends to $4$ for small $\epsilon$. In the proof, $W(x) = \log x - \log\log x + o(1) = \Theta(\log x)$ as $x\to\infty$ is the Lambert-$W$ function \cite{hoorfar2007approximation}.

\begin{theorem}\label{thm:lower}
    Fix some $0 < \delta < 1/10$. For any underlying dimension $d$, there exists a positive semi-definite matrix $\A\in\R^{d\times d}$, positive definite matrix $\M\in\R^{d\times d}$, and some $\epsilon_0>0$ so that for any $0 < \epsilon < \epsilon_0$, taking  $k = \lfloor\tfrac{4}{\epsilon^2}\log\tfrac{1}{\sqrt{8\pi}\delta} - \tfrac{2}{\epsilon^2}\log\log\tfrac{1}{\sqrt{8\pi}\delta}\rfloor$ guarantees the stability estimate $\Stab(\A,\M,k)$ returned by Algorithm \ref{alg:stab} fails to satisfy the equation
    \[
        \sqrt{1-\epsilon}\|\I - \M^{-1}\A\|_\f \leq \Stab(\A,\M,k) \leq \sqrt{1+\epsilon}\|\I - \M^{-1}\A\|_\f
    \]
    with probability at least $\delta$.
\end{theorem}
\begin{proof}
    If $Z \sim \mathcal{N}(0,1)$ is a standard normal random variable then
    \begin{equation}\label{eq:gaussian}
        \P(|Z| \geq t) = 2\P(Z > t) > \sqrt{\frac{2}{\pi}}\frac{t}{t^2+1}e^{-t^2/2} \geq \frac{1}{\sqrt{2\pi}}\frac{1}{t}e^{-t^2/2}
    \end{equation}
    by \cite{gordon1941values} for all $t \geq 1$. Setting the right hand side of Inequality \ref{eq:gaussian} to $2\delta$ and solving gives
    \begin{equation}\label{eq:gaussianlower}
        \P\bigl(|Z| \geq \sqrt{W(8^{-1}\pi^{-1}\delta^{-2})}\bigr) > 2\delta
    \end{equation}
    whenever $\sqrt{W(8^{-1}\pi^{-1}\delta^{-2})} \geq 1$, which is satisfied when $0 < \delta \leq 1/10$.
    
    Now let $\A = \I - \e_1\e_1^*$ and $\M=\I$, where $\e_1$ is the first standard basis vector. We can observe that
    \begin{equation}
        \|(\I - \M^{-1}\A)\q\|_2^2 = \|\e_1(\e_1^*\q)\|_2^2 = \q_1^2 \sim \chi^2
    \end{equation}
    if $\q$ is a standard Gaussian vector. In particular, the standard deviation of $\|(\I - \M^{-1}\A)\q\|_2^2$ is $\sigma=\sqrt{2}$. Thus since $\Stab(\A,\M,k)^2$ is a sample average of independent copies of these random variables, fixing $k = \lfloor \tfrac{2}{\epsilon^2}W(\tfrac{1}{8\pi\delta^2})\rfloor$ ensures
    \begin{align}
        \P(|\Stab(\A,&\M,k)^2 - 1| > \epsilon)\\&= \P(\tfrac{\sqrt{k}}{\sigma}|\Stab(\A,\M,k)^2-1| \geq \tfrac{\sqrt{k}\epsilon}{\sigma})\\
        &\geq \P\bigl(\tfrac{\sqrt{k}}{\sigma}|\Stab(\A,\M,k)^2-1| \geq \sqrt{W(8^{-1}\pi^{-1}\delta^{-2})}\bigr)\\
        &\to \P\bigl(|Z| \geq \sqrt{W(8^{-1}\pi^{-1}\delta^{-2})}\bigr) > 2\delta
    \end{align}
    by the Central Limit Theorem \cite[Thm. 1.3.2]{vershynin2018high} and Equation \ref{eq:gaussianlower} as $k\to\infty$. This implies the existence of an $\epsilon_0$ so that $0 < \epsilon < \epsilon_0$ ensures the relation
    \begin{equation}
        \sqrt{1-\epsilon}\|\I - \M^{-1}\A\|_\f \leq \Stab(\A,\M,k) \leq \sqrt{1+\epsilon}\|\I - \M^{-1}\A\|_\f
    \end{equation}
    fails with probability \emph{at least} $\delta$ under our relation defining $k$. The simpler condition on $k$ given in the statement of this result follows from the bound $W(x) \geq \log(x) - \log\log(x)$ for all $x\geq e$ from \cite[Thm. 2]{hoorfar2007approximation}.
\end{proof}

We point out that the above proof gives another confirmation of the conjecture in \cite{avron2011randomized} regarding the true asymptotics of the Gaussian trace estimator. The work in \cite{wimmer2014optimal} confirmed this conjecture to be true, but did so as a corollary of a general result regarding lower bounds for trace estimation algorithms. Our result above is much more direct.

\subsubsection{An Improvement in the Presence of a Clear Winner}
\label{sub:winner}

Algorithm \ref{alg:pick} is simple to implement and works well for selecting preconditioners of minimal stability, as we shall see in Section \ref{sec:experiments}. Nevertheless, if we are selecting between preconditioners where some are clearly worse than the optimal preconditioner in terms of stability, our method seems excessive. Intuitively, we should be able to tell that terrible preconditioners will not be optimal with very rudimentary information. Algorithm \ref{alg:improvement} presents such a revision to Algorithm \ref{alg:pick}, iteratively refining the stability estimates we have and filtering out any preconditioners as soon as we can be confident they will not be optimal. Note that the algorithm crucially relies on the bounds from Section \ref{sub:upper_bound}.

\begin{algorithm2e}[t]
    \KwData{A matrix $\A\in\F^{d\times d}$, $n$ candidate preconditioners $\M_1$, $\M_2$,$\ldots$,$\M_n\in\F^{d\times d}$, an accuracy parameter $0 < \epsilon < \tfrac{1}{2}$ and an acceptable failure probability $0 < \delta < 1$.}
    \KwResult{A preconditioner $\M_i$ for which the stability criterion $\|\I - \M_i^{-1}\A\|_\f$ is an $\epsilon$-multiplicative approximation to the minimum possible among the candidate preconditioners, except with probability at most $\delta$.}

    $\epsilon_\cur \leftarrow 1$
    
    $P \leftarrow \{1,2,\ldots,n\}$

    $T \leftarrow \lceil \log_2\tfrac{1}{\epsilon}\rceil$

    \For{$t = 1,2,\ldots, T$}{
        $\epsilon_\cur \leftarrow \epsilon_\cur / 2$

        $k \leftarrow \tfrac{6}{\epsilon_\cur^2}\log\tfrac{2T|P|}{\delta}$

        $S_i \leftarrow \Stab(\A,\M_i,k)$ for all $i\in P$

        $i^\star = \operatorname*{arg\,min}_{i\in P} S_i$

        $P \leftarrow \bigl\{i \in P : S_i \leq S_{i^\star}\sqrt{\tfrac{1+\epsilon_\cur}{1-\epsilon_\cur}} \bigr\}$
    }
    Return $\M_{i^\star}$

    \caption{An improvement to Algorithm \ref{alg:pick} when there is a relatively clear winner among the candidate preconditioners.}\label{alg:improvement}
\end{algorithm2e}

We can prove that Algorithm \ref{alg:improvement} is actually an improvement over Algorithm \ref{alg:pick} by making an anti-concentration assumption about the input stabilities.

\begin{theorem}\label{thm:pick_better}
    Let $\A\in\F^{d\times d}$ be an arbitrary matrix, $\M_1,\M_2,\ldots,\M_n\in\F^{d\times d}$ be invertible candidate preconditioners for $\A$, $0 < \epsilon < 1/2$, and $0 < \delta < 1$. Denoting $i^\star \in \operatorname{arg\,min}_{1\leq j \leq n}\|\I - \M_j^{-1}\A\|_\f$, we will write
    \[
        F(t) = \frac{1}{n}\biggl|\biggl\{ j : j\in\{1,2,\ldots,n\} \text{ and }\frac{\|\I - \M_j^{-1}\A\|_\f}{\|\I - \M_{i^\star}^{-1}\A\|_\f} \leq 1+t \biggr\}\biggr|
    \]
    for the (shifted) cumulative distribution function of the input relative stabilities. If $F(t) \leq ct$ uniformly over $t\in [\epsilon/2,2]$ for some positive constant $c$, then Algorithm \ref{alg:improvement} returns a preconditioner $\M_i$ satisfying
    \[
        \|\I - \M_i^{-1}\A\|_\f \leq \sqrt{\frac{1+\epsilon}{1-\epsilon}}\,\min_{1\leq j\leq n} \|\I - \M_j^{-1}\A\|_\f.
    \]
    with probability at least $1 - \delta$ using strictly fewer floating point operations than running $24n(1+\tfrac{2c}{\epsilon})\log\tfrac{2n}{\delta} + 24n(1+\tfrac{2c}{\epsilon})\log\log_2\tfrac{2}{\epsilon}$ iterations of the preconditioned conjugate gradients algorithm in $\A$ with the most expensive preconditioner $\M_j$ in terms of the number of floating point operations required to compute $\M^{-1}_j\y$ for input vectors $\y$.
\end{theorem}
\begin{proof}
    The same Bonferroni-correction argument from the proof of Theorem \ref{thm:pick} ensures that
    \begin{equation}\label{eq:approx_better}
        \sqrt{1-\epsilon_\cur}\,\|\I-\M_i^{-1}\A\|_\f \leq S_i \leq \sqrt{1+\epsilon_\cur}\,\|\I-\M_i^{-1}\A\|_\f,
    \end{equation}
    simultaneously for all $i\in P$ over the course of the algorithm, except with probability at most $\delta$. The rest of the proof will only rely on property \ref{eq:approx_better}, so everything we say will hold with this same probability.
    
    If $i_t^\star$ is the $i^\star$ set in step $t$ of the algorithm and $i^\star \in \operatorname*{arg\,min}_{1\leq i\leq n}\|\I - \M_i^{-1}\A\|_\f$ is in $P$ before the filtering at the end of step $t$, Equation \ref{eq:approx_better} implies
    \begin{equation*}
        S_{i^\star} \leq \sqrt{1+\epsilon_\cur}\|\I-\M_{i^\star}^{-1}\A\|_\f \leq \sqrt{\frac{1+\epsilon_\cur}{1-\epsilon_\cur}}\sqrt{1-\epsilon_\cur}\|\I-\M_{i_t^\star}^{-1}\A\|_\f \leq \sqrt{\frac{1+\epsilon_\cur}{1-\epsilon_\cur}}S_{i_t^\star}.
    \end{equation*}
    Thus, since $i^\star\in P$ initially, we know by induction that $i^\star\in P$ throughout the process of the entire algorithm. Now consider $P$ in the final step $t=T$ of Algorithm \ref{alg:improvement}. Since $i^\star \in P$,
    \begin{equation}
        \sqrt{1 - \epsilon_\cur}\|\I - \M_{i^\star_T}^{-1}\A\|_\f \leq S_{i^\star_T} \leq S_{i^\star} \leq \sqrt{1 + \epsilon_\cur}\|\I - \M_{i^\star}^{-1}\A\|_\f.
    \end{equation}
    Rearranging and realizing that $\epsilon_\cur = 2^{-\lceil \log_2\tfrac{1}{\epsilon}\rceil} \leq 2^{-\log_2\tfrac{1}{\epsilon}} = \epsilon$ at $t=T$ gives our desired approximation guarantee.

    Now we will exhibit the runtime bound by bounding $|P|$ at each step of Algorithm \ref{alg:improvement}. We claim that $|P| \leq 4cn 2^{-t}$ for all $t=1,2,\ldots,T$. To see this, note that if a candidate preconditioner $\M_j$ is retained after filtering in any step $t$ of the algorithm,
    \begin{equation}
        \|\I - \M_j^{-1}\A\|_\f \leq \frac{S_j}{\sqrt{1-\epsilon_\cur}} \leq \frac{\sqrt{1+\epsilon_\cur}}{1-\epsilon_\cur}S_{i^\star_t} \leq \frac{1+\epsilon_\cur}{1-\epsilon_\cur}\|\I - \M_{i^\star}^{-1}\A\|_\f.
    \end{equation}
    Thus the number of elements in $P$ just after step $t$ in the algorithm is at most $n F(4\epsilon_\cur) \leq 4cn 2^{-t}$ since $\tfrac{1+x}{1-x}\leq 1+4x$ for $0 \leq x \leq 1/2$. Our runtime bound follows from the sum
    \begin{equation}
        \sum_{t=1}^{T-1} |P_t|\frac{6}{(2^{-t})^2}\log\frac{2T|P_t|}{\delta} \leq \sum_{t=1}^{T-1} 4cn 2^{-t}\tfrac{6}{(2^{-t})^2}\log\tfrac{2nT}{\delta} \leq 24cn 2^T\log\tfrac{2nT}{\delta}
    \end{equation}
    where $P_t$ is the set $P$ during iteration $t$ of the algorithm. This gives the number of matrix-vector multiplies of the form $(\I - \M^{-1}\A)\q$ used by the algorithm \emph{after} the first step. To see the final form, add on the $24n\log\tfrac{2nT}{\delta}$ multiplies done during the first iteration $t=1$ and plug in $T = \lceil\log_2\tfrac{1}{\epsilon}\rceil \leq \log_2\tfrac{1}{\epsilon} + 1 = \log_2\tfrac{2}{\epsilon}$.
\end{proof}

The anti-concentration condition in Theorem \ref{thm:pick_better} intuitively asserts that the stabilities of the preconditioners do not cluster around the minimal stability. This is satisfied, for example, if at most some number $m$ of the candidate preconditioners have stability within a multiplicative factor $3$ of the optimal stability. The resulting constant $c = \tfrac{2m}{n\epsilon}$ gives an asymptotic runtime bound for Algorithm \ref{alg:improvement} of $\mathcal{O}(n\log\tfrac{n}{\delta} + n\log\log\tfrac{1}{\epsilon} + \tfrac{m}{\epsilon^2}\log\tfrac{n}{\delta} + \tfrac{m}{\epsilon^2}\log\log\tfrac{1}{\epsilon})$, decoupling the linear dependence in $n$ with the polynomial accuracy dependence on $1/\epsilon^2$. Such an improvement is serious when $n$ is moderately large; while this example is contrived many other distributions on input data satisfy the assumptions of Theorem \ref{thm:pick_better} with the same constant $c$. 

Of course, one would hope that Algorithm \ref{alg:improvement} does not perform poorly when the input data assumptions made in Theorem \ref{thm:pick_better} are not satisfied. For example, this would happen when all preconditioners have extremely similar performance, to the point that even our target accuracy $\epsilon$ cannot distinguish their stabilities. Luckily, a constant $c = \tfrac{2}{\epsilon}$ always works in Theorem \ref{thm:pick_better}, so Algorithm \ref{alg:improvement} never suffers more than a multiplicative $\mathcal{O}(\log\log\tfrac{1}{\epsilon})$ increase over Algorithm \ref{alg:pick} in number of floating point operations needed to select a preconditioner.

\section{Experiments}
\label{sec:experiments}

The present section will jointly test the hypothesis that preconditioner stability is a good forecast for preconditioner quality along with the performance of our algorithms. This is done by evaluating how well Algorithm \ref{alg:pick} can select a candidate preconditioner which minimizes the number of conjugate gradients iterations required to achieve some fixed approximation quality in various situations. In Section \ref{sec:experiments_sparse}, we detail one experiment of this type for generic real-world systems from the SuiteSparse Matrix Collection \cite{davis2011university}. Section \ref{sec:experiments_kernel} details our other numerical experiment, where we use Algorithm \ref{alg:pick} to contribute improvements to the growing literature on preconditioned solvers for kernel regression problems.

\subsection{Experiments with Sparse Systems}
\label{sec:experiments_sparse}

First we attempt a generic experiment on a collection of real-world sparse linear systems and simple preconditioners, testing how well the preconditioner chosen a-priori by Algorithm \ref{alg:pick} compares to the minimal-iterations preconditioner. For the target system $\A\x=\b$, we fix a sampled $\b \sim \mathcal{N}(\0,\I)$ for the entire experiment. The positive definite matrices $\A$ are taken from the SuiteSparse/University of Florida Sparse Matrix Collection \cite{davis2011university}. We include all matrices from the \verb+Boeing+ and \verb+GHS_psdef+ groups which have between 100,000 and 2,250,000 non-zero entries and are strictly positive definite.

We provide nine candidate preconditioners for Algorithm \ref{alg:pick} to select between, using block diagonal preconditioners in order to avoid existence issues of other common preconditioning methods \cite{benzi2002preconditioning}:
\begin{itemize}
    \item The first candidate preconditioner is the trivial preconditioner $\I$, which is equivalent to using no preconditioner at all.
    \item The preconditioner $\Blk_\ell$ denotes a block-diagonal pinching/truncation of the matrix $\A$ with block size $\ell$.
    \item The preconditioner $\RCM_\ell$ is the same block-diagonal pinching, but performed after a Reverse Cuthill-McKee ordering of the matrix \cite{cuthill1969reducing}.
\end{itemize}
To ensure uniqueness and clarity, blocking is performed by taking the matrix $\A\in\F^{d\times d}$ and constructing a block diagonal matrix $\M$ with blocks of the form $\A(m\ell:\min\{d,(m+1)\ell\},m\ell:\min\{d,(m+1)\ell\})$ for $m=0,1,2,\ldots$. Since $\A$ is positive definite, the resulting preconditioners $\M$ are also positive definite \cite[Ex. 2.2.1.(viii)]{bhatia2009positive}.

In Table \ref{table:test_preconditioners} we present the number of iterations the preconditioned conjugate gradients algorithm took for each test matrix and preconditioner pair. The algorithm was run until the approximate solution $\tilde\x$ satisfied $\|\A\tilde\x - \b\|_2 \leq 10^{-9}\|\b\|_2$. The number of iterations was limited to 50,000. Entries in Table \ref{table:test_preconditioners} achieving this iteration limit are overwritten with `---'. The conjugate gradients algorithm applied to the matrices \verb+bcsstk36+, \verb+bcsstk38+, \verb+msc23052+, and \verb+vanbody+ did not converge with any candidate preconditioner, so they are omitted in Table \ref{table:test_preconditioners}.

\begin{table}[t]
    \centering
    \caption{This table reports the number of iterations taken by the conjugate gradients algorithm to report an approximate solution $\tilde\x$ to the linear system $\A\x=\b$ for specified test matrices $\A$, a constant sampled standard normally distributed $\b\sim\mathcal{N}(\0,\I)$, and various candidate preconditioners.}
    \label{table:test_preconditioners}
    \setlength{\tabcolsep}{0.4em}
    \begin{tabular}{@{}llllllllll@{}}
        \toprule
        Matrix & \multicolumn{9}{c}{Conjugate Gradients Iterations With Various Preconditioners}   \\\cmidrule(r){1-1}\cmidrule(l){2-10}
        & $\I$ & $\Blk_1$ & $\Blk_{10}$ & $\Blk_{25}$ & $\Blk_{50}$ & $\Blk_{75}$ & $\Blk_{100}$ & $\RCM_{75}$ & $\RCM_{100}$\\\midrule
        \texttt{apache1} &3,538&3,513&3,286&3,283&3,270&3,265&3,269&3,710&3,693\\
        \texttt{crystm01} &122&54&39&34&30&27&27&24&21\\
        \texttt{crystm02} &138&54&38&35&34&29&30&24&24\\
        \texttt{crystm03} &143&54&38&34&33&30&29&25&24\\
        \texttt{cvxbqp1} &16,424&11,337&11,338&11,332&11,331&11,330&11,328&10,148&10,353\\
        \texttt{gridgena} &3,658&3,542&2,659&2,572&2,504&2,504&2,479&2,892&2,863\\
        \texttt{jnlbrng1} &139&131&126&126&125&125&125&130&130\\
        \texttt{minsurfo} &94&88&64&63&63&62&62&88&88\\
        \texttt{msc10848} &---&5,659&3,791&3,028&2,793&2,656&2,628&2,192&2,092\\
        \texttt{obstclae} &66&65&49&48&47&47&47&65&65\\
        \texttt{oilpan} &48,291&28,065&12,804&8,167&5,476&4,992&4,127&4,757&4,433\\
        \texttt{torsion1} &66&65&49&48&47&47&47&65&65\\
        \texttt{wathen100} &327&45&44&44&44&44&44&42&42\\
        \texttt{wathen120} &378&45&45&45&44&44&44&43&43\\
        \bottomrule\\
    \end{tabular}
\end{table}

Even though larger block sizes $\ell$ ought to create better approximations of the original matrix, there are situations when smaller block sizes result in fewer conjugate gradients iterations. Similarly, there are some situations when the original ordering of the data is preferable over the Reverse Cuthill-McKee ordering, and vice-versa. As a result, it is unclear a-priori which preconditioner one should choose to solve the linear system, and this is why someone might wish to use Algorithm \ref{alg:pick} to automate that choice.

We test this use of Algorithm \ref{alg:pick} under two parameter settings $k=10$ and $k=50$. Algorithm \ref{alg:pick} is run for 1,000 independent trials for each matrix-preconditioner-$k$ pairing. After the fact, we compare the number of iterations of preconditioned conjugate gradients would be necessary when using the recommendation of Algorithm \ref{alg:pick} relative to the minimal number of iterations possible if we knew in advance how many iterations each preconditioner would use.

\subsubsection{Results}

The results of our generic real-world-use experiment are presented in Table \ref{table:sparse_result}. Every cell is an approximation ratio, i.e. the number of iterations an algorithm for selecting preconditioners took divided by the minimal number of iterations possible using our set of candidate preconditioners. As such, an entry of 1.00 is optimal and represents the minimal-number-of-iterations preconditioner being correctly selected. The column `Worst-Case' reports the approximation ratio if one deterministically selected the maximal-number-of-iterations preconditioner in each setting. The column `Random' reports the expected approximation ratio if one were to select a candidate preconditioner from Table \ref{table:test_preconditioners} uniformly at random. The columns corresponding to Algorithm \ref{alg:pick} gives statistics of the empirical distribution of approximation ratios seen over the 1,000 independent trial runs of the method.

\begin{table}[t]
    \centering
    \caption{This table summarizes the performance of Algorithm \ref{alg:pick} for each matrix in Table \ref{table:test_preconditioners}, reporting statistics of the empirical number of iterations given by the algorithm compared to picking the worst-possible preconditioner (in terms of number of CG iterations) or choosing arbitrarily at random. Since the conjugate gradients algorithm did not converge for the matrix \texttt{msc10848} with no preconditioner, the `Worst-Case' and `Random' columns are \emph{lower} bounds for their true values in that row only.}
    \label{table:sparse_result}
    \setlength{\tabcolsep}{0.5em}
    \begin{tabular}{@{}lcccccccc@{}}%llllllll@{}}
        \toprule
        Matrix & Worst-Case & Random & \multicolumn{6}{c}{Algorithm \ref{alg:pick} Approximation Ratio} \\ \cmidrule(r){1-1} \cmidrule(lr){2-2} \cmidrule(lr){3-3} \cmidrule(l){4-9}
        &&&\multicolumn{3}{c}{$k=10$}&\multicolumn{3}{c}{$k=50$}\\\cmidrule(lr){4-6}\cmidrule(l){7-9}
        &&&Min&Mean&Max&Min&Mean&Max\\\midrule
        \texttt{apache1}  &1.14&1.05&1.00&1.00&1.00&1.00&1.00&1.00\\
        \texttt{crystm01}  &5.81&2.00&1.00&1.00&1.00&1.00&1.00&1.00\\
        \texttt{crystm02}  &5.75&1.88&1.00&1.00&1.00&1.00&1.00&1.00\\
        \texttt{crystm03}  &5.96&1.90&1.00&1.00&1.00&1.00&1.00&1.00\\
        \texttt{cvxbqp1}  &1.62&1.15&1.12&1.12&1.12&1.12&1.12&1.12\\
        \texttt{gridgena}  &1.48&1.15&1.00&1.00&1.01&1.00&1.00&1.00\\
        \texttt{jnlbrng1}  &1.11&1.03&1.04&1.04&1.04&1.04&1.04&1.04\\
        \texttt{minsurfo}  &1.52&1.20&1.00&1.00&1.00&1.00&1.00&1.00\\
        \texttt{msc10848}  &\hphantom{2}23.90&3.97&1.00&1.00&1.00&1.00&1.00&1.00\\
        \texttt{obstclae}  &1.40&1.18&1.00&1.00&1.00&1.00&1.00&1.00\\
        \texttt{oilpan}  &\hphantom{2}11.70&3.26&1.00&1.09&1.15&1.07&1.08&1.15\\
        \texttt{torsion1}  &1.40&1.18&1.00&1.00&1.00&1.00&1.00&1.00\\
        \texttt{wathen100}  &7.79&1.79&1.00&1.00&1.00&1.00&1.00&1.00\\
        \texttt{wathen120}  &8.79&1.89&1.00&1.00&1.00&1.00&1.00&1.00\\
        \bottomrule\\
    \end{tabular}
\end{table}

For 10 of the 14 test matrices reported, setting $k=10$ always picks the optimal preconditioner for the problem across every one of the 1,000 trials. If we take $k=50$, this happens for 11 of the 14 test matrices. Moreover, even when the accuracy parameter $k=10$, the returned preconditioner never needs more than 15\% iterations over than the optimal choice.

One might wonder if taking $k$ to be even larger would result in approximation ratios concentrating more uniformly at the ideal 1.00 mark. This will not happen in general, and is where the good-proxy hypothesis is put to the test. For the \verb+oilpan+ matrix, increasing $k$ from $10$ to $50$ \emph{raises} the best-seen approximation ratio given by Algorithm \ref{alg:pick} from 1.00 to 1.07.  Increasing $k$ causes the preconditioner returned by Algorithm \ref{alg:pick} to concentrate further around the true minimal-stability preconditioner (see Theorem \ref{thm:pick},) and so this implies that the preconditioner stability criterion itself is not perfect and will not in general forecast the exact preconditioner resulting in the minimal number of conjugate gradients iterations.

\subsection{Experiments with Kernel Regression Preconditioners}
\label{sec:experiments_kernel}

This section will show that Algorithm \ref{alg:pick} can turn two simple preconditioners for the standard kernel regression problem into a robust, state-of-the-art preconditioning method. As a corollary of this investigation, we exhibit how Algorithm \ref{alg:pick} performs well in situations when the `minimal accuracy' criterion for selecting preconditioners fails, something left unanswered in the previous experiment.

\subsubsection{A Quick Review}

Kernel regression is a common statistical technique for nonlinear regression. In this setting, we have a dataset $\{(\x_1,y_1),(\x_2,y_2),\ldots,(\x_d,y_d)\}$ consisting of $\x_i \mapsto y_i$ mappings from Euclidean space $\R^d$ to the real line $\R$. We wish to find coefficients $\aalpha \in \R^d$ so that the functional mapping
\begin{equation}
    \x \mapsto f(\x) = \sum_{i=1}^d \aalpha_i k(\x,\x_i)
\end{equation}
faithfully represents the empirical mapping in the sense that $f(\x_i) \approx y_i$. In general, $k(\x,\y)$ is just required to be a positive definite kernel, but in our experiment, we will only use the squared exponential kernel $k(\x,\y) = \exp(-\tfrac{\|\x-\y\|_2^2}{2\ell^2})$, parametrized by the length-scale $\ell$ which controls the derivative of the model $f(\x)$. The coefficients $\aalpha$ are found by solving the system
\begin{equation}
    \aalpha = (\K + \sigma_n^2\I)^{-1}\y
\end{equation}
where the positive definite Gram matrix $\K_{ij} = k(\x_i,\x_j)$, the output vector $\y = (y_1,y_2,\ldots,y_d)$, and the noise standarad deviation $\sigma_n>0$ is used for regularization so that the model $f(\x)$ fits well on out-of-sample data. In almost all kernel regression problems, $\K$ and hence $\K+\sigma_n^2\I$ are dense. See \cite{williams2006gaussian} for more background on this model and associated inference procedure.

\subsubsection{Related Work}

This experiment will test a preconditioning procedure for solving the linear system $(\K+\sigma_n^2\I)\aalpha = \y$ via conjugate gradients. There has been a recent interest in this general iterative framework for kernel regression, which largely focuses on creating quality preconditioners for the Gram matrices $\K$ \cite{avron2017faster,cutajar2016preconditioning,halko2011finding,rudi2017falkon,rudi2018fast}. This is necessary since for reasonable parameter settings even $\K+\sigma_n^2\I$ is often ill-conditioned.

Cutajar et al. \cite{cutajar2016preconditioning} performed some initial leg-work in this area, proposing eight candidate preconditioners. These preconditioners include a block-diagonal approximation of $\K+\sigma_n^2\I$, adding a larger regularizer $\sigma_n^2$ and solving recursively, a Nystr\"om approximation of the Gram matrix using $\sqrt{n}$ data points as inducing points chosen uniformly at random, a coupling of the Nystr\"om approximation with a block-diagonal approximation, or replacing $\K$ with an optimal low-rank factorization which can be computed via a randomized SVD \cite{halko2011finding} or the Lanczos method \cite[Sec. 10.1]{golub2012matrix}. Both \cite{cutajar2016preconditioning} and the work \cite{avron2017faster} of Avron et al. use the Fourier features method of Rahimi and Recht \cite{rahimi2008random} to create a preconditioner which replaces $\K$ with a sketched version $\tilde\K$. The latter paper \cite{avron2017faster} also proposes using the \textsc{TensorSketch} method of \cite{pagh2013compressed} for creating a sketched preconditioner when using the polynomial kernel $k(\x,\y)=(\x^*\y)^q$, though the necessary sketching dimension in their upper bounds is exponential in $q$. The work of Rudi et al. \cite{rudi2017falkon} also uses the Nystr\"om-based preconditioner like \cite{cutajar2016preconditioning}, combining it with other computational tricks. This Nystr\"om-based approach was refined with approximate leverage score sampling in more recent work by Rudi et al. \cite{rudi2018fast}.

An empirical issue within many of the above methods is illustrated perfectly in Figure 1 of \cite{cutajar2016preconditioning}. For every preconditioner presented therein, there exist parameter settings for which using no preconditioner results in fewer iterations than using the preconditioner when solving for $\aalpha$ via conjugate gradients. As such, it is unclear how one would choose a preconditioner in practice which makes these solvers work well in practice.

\subsubsection{Two Simple Geometrically Driven Preconditioners}

Here we detail the two candidate preconditioners which we will use in our experiments. They both utilize a geometrically-motivated reordering of the data to empirically achieve superior performance to the preconditioners of \cite{cutajar2016preconditioning} for many settings of the kernel parameters $\sigma_n$ and $\ell$.

The first preconditioner is a simple block diagonal pinching of a reordering of the data. The kernel regression model under the squared exponential kernel effectively asserts that points nearby in $\ell^2$ ought to have similar outputs $y$. If the input data is highly clustered in $\ell^2$, our model then ought to largely ignore points from different clusters when considering a point in some cluster. The first preconditioning algorithm turns this `ought to' statement directly into an approximation of the Gram matrix $\K$. We first cluster the data $\{\x_1,\x_2,\ldots,\x_d\}$ in $\ell^2$ via the \verb/k-means/ or \verb/k-means++/ \cite{arthur2007k} algorithm with $c = \lceil\sqrt{d}\rceil$ clusters, constructing a permutation matrix $\PP$ that places points in the same cluster next to each other on the Gram matrix. At this point, we precondition the re-ordered system $(\PP\K\PP^* + \sigma_n^2\I)\PP\aalpha = \PP\y$ by creating a block-diagonal pinching of the re-ordered matrix $\PP\K\PP^*$ where each block corresponds to the points within a cluster. The resulting preconditioner is that pinching $\hat\K$ plus the true noise term $\sigma_n^2\I$.

The second preconditioner is a slightly more complex version of the first. After computing the permuted matrix $\PP\K\PP^*$, we compute a truncated rank-$r$ approximation $\U\Lambda\U^*$ of $\PP\K\PP^*$ where $\Lambda\in\R^{r\times r}$ is diagonal and $\U\in\R^{d\times r}$ has orthonormal columns. At this point we compute the same block diagonal pinching $\tilde\EE$ of the error in approximation $\EE = \PP\K\PP^* - \U\Lambda\U^*$. The resulting preconditioner is then $\U\Lambda\U^* + \tilde\EE + \sigma_n^2\I$.

To show that these are feasible to use as preconditioners, it suffices to consider the latter since it reduces to the former when $r=0$. If $r$ is a constant, we can solve systems in this preconditioner using the Woodbury identity \cite[Sec. 2.1.4]{golub2012matrix} in $\mathcal{O}(n^2)$ floating point operations under the assumption that the cluster sizes are all within a constant factor of each other. This is after computing a one-time Cholesky decomposition of the block-diagonal pinching in $\mathcal{O}(n^2)$ floating point operations. Similarly, computing the low-rank factorization takes $\mathcal{O}(n^2)$ floating point operations using either the Implicitly-Restarted Lanczos method \cite{sorensen1997implicitly} or a Randomized SVD \cite{halko2011finding}, though for higher ranks $r$ the latter method is preferable. Since matrix-vector multiplies with $\K + \sigma_n^2\I$ take $\mathcal{O}(n^2)$ floating point operations, these preconditioners do not raise the per-iteration complexity of the conjugate gradients method.

Moreover, if we fix the resulting sparsity pattern of the preconditioner, the former preconditioner exactly minimizes the accuracy $\|\M - \PP\K\PP^* + \sigma_n^2\I\|_\f$ over all matrices with the same sparsity pattern. Since the identity matrix $\I$ also has this sparsity pattern, we would always choose the preconditioner over the identity matrix if using the accuracy criterion, and a similar logic shows the same holds for the latter preconditioner.

The work of Cutajar et al. \cite{cutajar2016preconditioning} already suggests that block-diagonal approximations to kernel matrices can be faithful for small length-scales $\ell$, and that low-rank approximations can be faithful for large length-scales $\ell$. Following this helpful research, the main insight in our proposed preconditioners is that permutating the data in a way that points close together in Euclidean space are close on the matrix ensures a block-diagonal approximation is even more faithful. The permutation we use relies on the Euclidean geometry of our kernel, and other permutations would need to be used for kernels other than the squared exponential kernel.

\subsubsection{Experimental Design}

We consider three datasets along with various parameter settings, aiming to see how many iterations the conjugate gradients algorithm takes with each of our proposed preconditioners, as well as using Algorithm \ref{alg:pick}, in comparison to using no preconditioner at all. This experiment is identical to one from \cite{cutajar2016preconditioning} except with different preconditioners, and ideally our preconditioned solvers significantly reduce the number of iterations used by the standard conjugate gradients algorithm.

The datasets have names \verb+Concrete+, \verb+Power+, and \verb+Protein+, and are identically the same as the data in \cite{cutajar2016preconditioning} pulled from the UCI Machine Learning Repository \cite{asuncion2007uci}. The \verb+Concrete+ dataset consists of $d=$ 1,029 data points in $\R^8$. The \verb+Power+ dataset consists of $d=$ 9,567 data points in $\R^4$. The \verb+Protein+ dataset consists of $d=$ 45,729 data points in $\R^9$.

For each of these datasets, and each pair of parameters chosen from $\ell \in \{10^{-3},$ $10^{-2},\ldots, 10^{2}\}$ and $\sigma_n^2 \in \{10^{-2}, 10^{-4}, 10^{-6}\}$, we construct a kernel system $(\K + \sigma_n^2\I)\aalpha = \y$. This system is solved using conjugate gradients with no preconditioner, the geometric preconditioner with no low-rank approximation, and the geometric preconditioner with a rank $r=25$ low-rank approximation. We also solve the system using the preconditioner chosen by one run of Algorithm \ref{alg:pick} among no preconditioner, the purely block-diagonal geometric preconditioner, and the rank $r=25$ low-rank approximation-based geometric preconditioner, using an accuracy parameter $k=10$. We also attempt using Algorithm \ref{alg:pick} with the same $k=10$ if we restrict the choice to the two geometric preconditioners, ruling out the use of no preconditioner. In solving these systems, we record the number of conjugate gradients iterations needed to achieve a residual norm less than $10^{-5}\sqrt{d}$ as in \cite{cutajar2016preconditioning}; a relative tolerance of $10^{-15}\|\y\|_2$ is also specified, though this is vacuous in comparison to the absolute tolerance. The solver is stopped after 10,000 iterations if the residual has not converged to within tolerance by then. The low-rank approximations are computed via \verb+ARPACK+ \cite{lehoucq1998arpack} with a tolerance parameter of $10^{-5}$.

\subsubsection{Results}

\begin{figure}[t]
    \centering
    \caption{This figure presents the relative improvement of using our proposed preconditioners, or the one automatically chosen by Algorithm \ref{alg:pick}, with respect to using no preconditioner at all. Each individual matrix corresponds to a specific preconditioner and dataset pair. Each row gives the value of $\log\sigma_n^2$ used in the experiment, whereas each column corresponds to $\log\ell$. The absence of red cells in the result matrices corresponding to `Our Method' indicates significant improvement over the results in \cite{cutajar2016preconditioning}.\vspace{0.5em}}
    \includegraphics[width=0.95\linewidth]{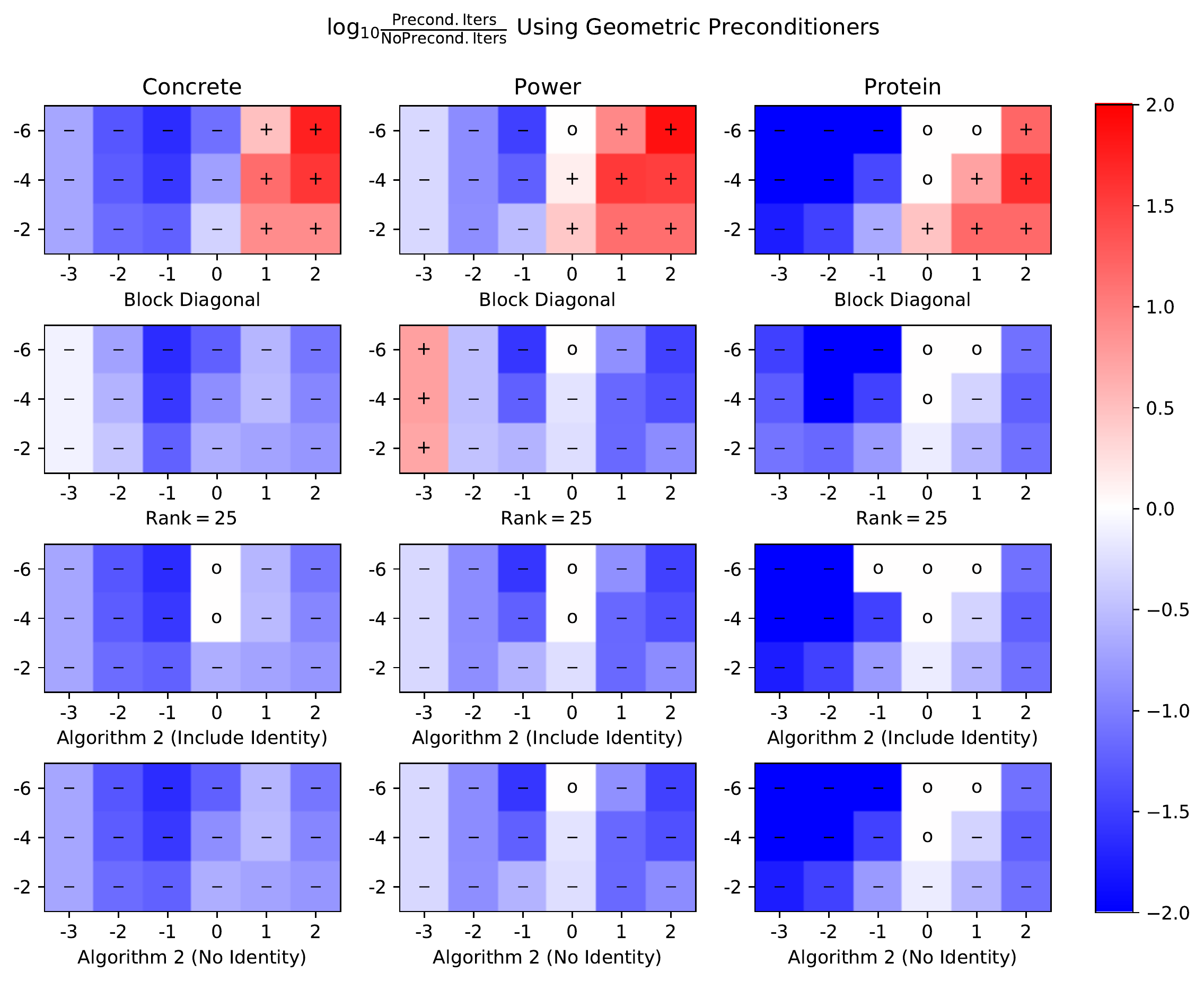}
    \label{fig:geo}
\end{figure}

Figure \ref{fig:geo} illustrates the relative improvement different preconditioning schemes have over using no preconditioner for each dataset and parameter combination. Each cell gives the logarithm of the ratio of the preconditioned conjugate gradients iterations to the non-preconditioned conjugate gradients iterations, i.e. the order of magnitude of the improvement granted by using the preconditioner. Accordingly, negative values (blue or `--') represent improvement through using the preconditioner, while positive values (red or `+') correspond to the preconditioned system requiring \emph{more} iterations than using no preconditioner at all. Five of the cells for the \verb+Protein+ dataset with the purely block diagonal geometric preconditioner have relative improvements of more than two orders of magnitude. Another three preconditioners using a low-rank approximation with the \verb+Protein+ dataset have this property. In spite of this, we restrict the visual range of the plot from $-2$ to $2$ to allow Figure \ref{fig:geo} to be compared easily to the identical presentation in Figure 1 of \cite{cutajar2016preconditioning}. No cell values exceed $2$.

First we comment purely on the performance of the two geometrically motivated preconditioners. The main take-away is that the geometric permutation based on the \verb+k-means+ algorithm appears to truly help in creating a faithful preconditioner. As evidence, we can point to the fact that the simple geometric block-diagonal preconditioner gives, for five different parameter settings with the \verb+Protein+ dataset, a relative improvement better than every single preconditioner-parameters-dataset pair in \cite{cutajar2016preconditioning}. Phrased differently, at these parameter settings the number of iterations drops from 189, 111, 2,345, 618, and 10,000 (did not converge) to 1, 1, 3, 3, and 94 iterations, respectively. Moreover, the geometric preconditioner using a low-rank approximation for the \verb+Concrete+ dataset always outperforms using no preconditioner, something no preconditioner proposed in \cite{cutajar2016preconditioning} can do. These improvements are genuine and stark, and again achieved by an extremely simple method just by relying on geometry.

Of course, one can rightfully point out that the block diagonal pinching is not robust as a preconditioner, just like many methods from \cite{cutajar2016preconditioning}. This is true; the block diagonal approximation works well for small length scales $\ell$, as in these circumstances dependencies $\K_{ij}$ between far away data points $\x_i$ and $\x_j$ are shrunk, resulting in a genuine clustering of the underlying data where the intuition we used in justifying the preconditioner carries through. For large $\ell$, the block diagonal preconditioner performs poorly because the matrix $\K$ looks more uniform and doesn't have a genuine clustered structure. Luckily, the more sophisticated preconditioners with added rank-25 terms perform well in precisely this regime, as the low-rank term can capture uniform structure in the Gram matrix $\K$. While this complicated preconditioner is not perfect, it is more robust to parameter changes than the analogous SVD-based preconditioner from \cite{cutajar2016preconditioning}. Between our two candidate preconditioners, at least one provides a performance boost over non-preconditioned conjugate gradients for every dataset and parameter setting chosen. Such a claim cannot be said about any pair of preconditioners in \cite{cutajar2016preconditioning}.

Since we have two quality preconditioners, each performing admirably in opposing parameter regimes, we might hope to get the best of both worlds by forecasting via Algorithm \ref{alg:pick} which one will perform better than using no preconditioner and solving the system with that resulting preconditioner. This approach does quite well, as we can see in Figure \ref{fig:geo}. While Algorithm \ref{alg:pick} does not always pick the best preconditioner in terms of minimizing the number of conjugate gradients iterations, it never selects a preconditioner which performs worse than using no preconditioner. That said, a preconditioner resulting in an exactly minimal number of iterations is chosen over 80\% of the time if the `use no preconditioner' option is included, and over 40\% of the time the preconditioner ranking induced by our stability estimates exactly corresponds to the ranking induced by the true iteration count. If we exclude the `use no preconditioner' option, which corresponds to an a-priori understanding that at least one of the geometric preconditioners works well, the former statistic jumps from 80\% to an impressive 98.1\%. This `all blue' plot which represents a robust preconditioner regime can not be found using the techniques of \cite{cutajar2016preconditioning}. Moreover, the algorithm was able to return the advice `use no preconditioner' in the face of uncertainty instead of suggesting the use of a poor preconditioner. This fact alone is highly desirable for the practitioner.

To confirm the importance of this paper, it is necessary to show that our method performs well when the computationally simple accuracy method does not. As mentioned when detailing the construction of these preconditioners, the accuracy criterion would never choose the `use no preconditioner' option over one of the geometric preconditioner. If we were just looking at the purely block-diagonal geometric preconditioner versus the `use no preconditioner option', the accuracy criterion would result in a poor preconditioner (higher number of iterations than possible) exactly a third of the time with the \verb+Concrete+ dataset. Of these times that the accuracy method fails, the estimated stability criterion succeeds exactly half of the time. For the \verb+Power+ dataset, the accuracy method fails 44.4\% of the time, but our estimated stability criterion succeeds in a quarter of these cases. While this behavior is not universal, it indicates that our method can be a crucial help when standard tools fail.

Finally, it is important to point out that in this setting, Algorithm \ref{alg:pick} performed computation commensurate with taking 30 steps of conjugate gradients. Since in over half of the parameter-dataset pairs the non-preconditioned conjugate gradients algorithm took more than five times this number of iterations, and our method can in most situations reduce that full-solution cost significantly, this initial cost is acceptable.

\section{Conclusions}
\label{sec:conclusions}

We have created a method to make the conjugate gradients algorithm friendlier to the practitioner. In particular, we took a quality forecast of preconditioner quality previously thought of as unusable in this arena, preconditioner stability, and presented a randomized algorithm which can quickly compute this quantity and use it to select a quality preconditioner. Our methods are motivated and justified heavily by theory and backed up by empirical evidence which suggests its applicability in the real world. In particular, Algorithm \ref{alg:pick} allowed us to create the first practical preconditioning system for kernel regression which is reported to never run fewer iterations than using no preconditioner at all in standard experiments; this is accomplished with surprisingly little leg-work.

Our work raises some important theoretical questions which would be ripe for future work. Most notably, it would be helpful to determine the fundamental limits of using preconditioner stability as a proxy for preconditioner quality with regard to different iterative methods. For example, it would be interesting to know if there are arbitrarily long sequences of preconditioners for some fixed positive definite matrix for which the order induced by the preconditioner-stability criterion is the opposite of the order induced by the condition number.

Finally, we note that our framing in Section \ref{sub:computing_stability} allows us to interpret the problem of finding a good preconditioner as a multi-armed bandit problem. Algorithms and lower bounds from this research area -- see \cite{jamieson2013finding} for example -- would then be informative to our problem of interest. As before, this investigation is left for future work.

\section*{Acknowledgments} We would like to thank Nicholas Pippenger for fruitful discussions early in the process of this research, as well as Karl Wimmer for helping us understand the main result in \cite{wimmer2014optimal}. We also appreciate the support of Harvey Mudd College in the form of computational resources used during this work.

\bibliographystyle{siamplain}
\bibliography{references}
\end{document}